\documentclass[3p, authoryear, square]{elsarticle}
\usepackage{amsthm,amssymb,amsmath}
\allowdisplaybreaks
\usepackage{bm}
\usepackage{chngcntr}
\counterwithin{equation}{section}
\theoremstyle{definition}
\newtheorem{theorem}{Theorem}[section]
\newtheorem{definition}{Definition}[section]

\usepackage[usenames, dvipsnames]{color}
\usepackage{hyperref}
\hypersetup{
  colorlinks=true,
  pdfinfo={
    CreationDate={D:20180816104534},
    ModDate={D:20180816104534},
  },
}

\newcommand{\reals}{\mathbf{R}}
\newcommand{\bigO}[1]{\mathcal{O}\left(#1\right)}
\newcommand{\eps}{\varepsilon}

\makeatletter
\def\ps@pprintTitle{%
  \let\@oddhead\@empty
  \let\@evenhead\@empty
  \def\@oddfoot{\footnotesize\itshape
    Published in Computer Aided Geometric Design (\cite{Hermes2019})
      \hfill November 8, 2019}%
  \let\@evenfoot\@oddfoot}
\makeatother

\begin{document}
%% H/T: https://tex.stackexchange.com/a/263503/32270
\hypersetup{
  urlcolor=MidnightBlue,
  linkcolor=MidnightBlue,
  citecolor=ForestGreen,
}

\begin{frontmatter}

\title{A 2-Norm Condition Number for B\'{e}zier Curve Intersection}
\author[djh]{Danny Hermes}\ead{dhermes@berkeley.edu}
\address[djh]{UC Berkeley, 970 Evans Hall \#3840, Berkeley, CA 94720-3840 USA}

\begin{abstract}
We present a condition number
of the intersection of two B\'{e}zier curves.
\end{abstract}


\end{frontmatter}

\section{Introduction}

The problem of intersecting two planar B\'{e}zier curves is an important
one in Computer Aided Geometric Design (CAGD).
Many intersection algorithms have been described in the literature, both
geometric (\cite{Sederberg1986, Sederberg1990, Kim1998}) and algebraic
(\cite{Manocha:CSD-92-698}). Though the general convergence properties
of these algorithms have been studied (e.g. \cite{Schulz2009}),
no condition number for the intersection problem has been described in
the CAGD literature\footnote{As far as the author has been able to tell.
In many CAGD textbooks, there is a long review of methods for intersecting
two planar B\'{e}zier curves (e.g. \cite{Farin2001, SederbergNotes}) but no
mention of conditioning.}.

There are more generic condition numbers for rational polynomial systems
(\cite{Herman2015}) or nonlinear algebraic systems
(\cite[Chapter~25, Section 25.4]{Higham2002}). However, the condition
numbers with an algebraic focus (rather than an analytic one) often
require too much computation to be useful. The numerical analytic condition
numbers are in some ways too general to be useful for planar B\'{e}zier
curve intersection.

In this paper, we describe a simple relative root condition number for this
intersection problem. Since tangent intersections are to transversal
intersections as multiple roots are to simple roots of a function, this
condition number is infinite for non-transversal intersections.
We present a few examples verifying that the condition number increases as
a family of intersections approach an ill-behaved intersection.

\section{Preliminaries}

Throughout, we will refer to a parametric polynomial
plane curve given by
\begin{equation}
b(s) = \sum_{j = 0}^n \bm{b}_j B_{j, n}(s)
\end{equation}
as a \emph{B\'{e}zier curve},
where \(B_{j, n}(s) = \binom{n}{j} (1 - s)^{n - j} s^j\) is a Bernstein
polynomial.
When the parameter \(s \in \left[0, 1\right]\), the coefficients
\(B_{j, n}(s) \in \left[0, 1\right]\) as well and the evaluation is a convex
combination of the \emph{control points} \(\bm{b}_j \in \reals^2\).

An intersection of two planar curves \(b_0(s)\) and \(b_1(t)\) corresponds to
a root \(\left[\begin{array}{c}
\alpha \\ \beta \end{array}\right]\) of the function
\begin{equation}
F(s, t) = b_0(s) - b_1(t).
\end{equation}
Note that
\(F: \reals^2 \longrightarrow \reals^2\).

Each component \(x(s)\) and \(y(s)\) of a B\'{e}zier curve is a
polynomial in Bernstein form. For such a polynomial
\begin{equation}
p(s) = \sum_{j = 0}^n p_j B_{j, n}(s)
\end{equation}
the (absolute) condition number of evaluation is
\begin{equation}\label{eq:p-tilde}
\widetilde{p}(s) = \sum_{j = 0}^n \left|p_j\right| B_{j, n}(s)
\end{equation}
(\cite{Farouki1987}) when the parameter \(s\) is in the unit interval.

\section{Conditioning of Generic Root-finding}\label{sec:generic}

Consider a smooth function \(F: \reals^n \longrightarrow \reals^n\)
with Jacobian \(F_{\bm{x}} = J\). We want to consider a special class of
functions of the form \(F\left(\bm{x}\right) = \sum_j c_j
\phi_j\left(\bm{x}\right)\) where the basis
functions \(\phi_j\) are also smooth functions
\(\reals^n \longrightarrow \reals^n\)
and each \(c_j \in \reals\). We want to consider the effects on a root
\(\bm{\alpha} \in \reals^n\) of a perturbation in one of the
coefficients \(c_j\). We examine the perturbed function
\begin{equation}
G\left(\bm{x}, \delta\right) = F\left(\bm{x}\right) +
\delta \phi_j\left(\bm{x}\right).
\end{equation}
Since \(G\left(\bm{\alpha}, 0\right) = \bm{0}\), if \(J^{-1}\) exists at
\(\bm{x} = \bm{\alpha}\),
the implicit function theorem tells us that we can define
\(\bm{x}\) via
\begin{equation}
G\left(\bm{x}\left(\delta\right), \delta\right) = \bm{0}.
\end{equation}
Taking the derivative with respect to \(\delta\) we find that
\(\bm{0} = G_{\bm{x}} \bm{x}' + G_{\delta}\). Plugging in
\(\delta = 0\) we find that \(\bm{0} = J\left(\bm{\alpha}\right) \bm{x}' +
\phi_j\left(\bm{\alpha}\right)\), hence we
conclude that
\begin{equation}
\bm{x}\left(\delta\right) = \bm{\alpha} - J^{-1}\left(\bm{\alpha}\right)
  \phi_j\left(\bm{\alpha}\right) \delta + \bigO{\delta^2}.
\end{equation}
This gives
\begin{equation}
\frac{\left \lVert J^{-1}\left(\bm{\alpha}\right)
  \phi_j\left(\bm{\alpha}\right) \right \rVert}{
  \left \lVert \bm{\alpha} \right \rVert}.
\end{equation}
as the relative condition number for a perturbation in \(c_j\).
By considering perturbations in \emph{all} of the coefficients:
\(\left|\delta_j\right| \leq \eps \left|c_j\right|\), a similar analysis
gives a root function
\begin{equation}
\bm{x}\left(\delta_0, \ldots, \delta_n\right) = \bm{\alpha} -
  J^{-1}\left(\bm{\alpha}\right) \sum_{j = 0}^n \delta_j
  \phi_j\left(\bm{\alpha}\right) + \bigO{\eps^2}.
\end{equation}
With this, we can define a root condition number

\begin{definition}\label{defn:abstract-cond-num}
For a smooth function \(F\left(\bm{x}\right) = \sum_j c_j
\phi_j\left(\bm{x}\right)\) parameterized by the coefficients
\(c_j\) with root \(\bm{\alpha}\) and Jacobian
\(J\left(\bm{\alpha}\right)\), we define a relative root condition
number
\begin{equation}
\kappa_{\bm{\alpha}} =
  \limsup_{\eps \to 0} \frac{\left \lVert\delta \bm{\alpha}
  \right \rVert / \eps}{\left \lVert\bm{\alpha}\right \rVert} =
  \lim_{\eps \to 0} \left(\sup_{\left|\delta_j\right| \leq
  \eps \left|c_j\right|} \frac{\left \lVert
  J^{-1}\left(\bm{\alpha}\right) \sum_j \delta_j
  \phi_j\left(\bm{\alpha}\right) \right \rVert / \eps}{
  \left \lVert\bm{\alpha}\right \rVert}\right)
\end{equation}
where \(\bm{\alpha} + \delta \bm{\alpha}\) is a root of the perturbed
function \(\sum_j (c_j + \delta_j) \phi_j\left(\bm{x}\right)\).
\end{definition}

In \cite[Chapter~25, Section 25.4]{Higham2002} a similar definition is
given. Instead of bounding the perturbations component-wise, it bounds
the entire perturbation vector \(\bm{\delta}\)
\begin{equation}
  \lim_{\eps \to 0} \left(\sup_{\left \lVert \bm{\delta} \right \rVert \leq
  \eps \left \lVert \bm{c} \right \rVert} \frac{\left \lVert
  J^{-1}\left(\bm{\alpha}\right) \sum_j \delta_j
  \phi_j\left(\bm{\alpha}\right) \right \rVert / \eps}{
  \left \lVert \bm{\alpha} \right \rVert}\right).
\end{equation}
It is possible to rewrite \(\sum_j \delta_j \phi_j\left(\bm{\alpha}\right) =
F_{\bm{c}} \bm{\delta}\) where \(F_{\bm{c}} =
\left[\frac{\partial F_i}{\partial c_j}\right]_{ij} =
\left[\begin{array}{c c c} \phi_0\left(\bm{\alpha}\right) & \cdots &
\phi_n\left(\bm{\alpha}\right) \end{array}\right]\) is the Jacobian of \(F\)
with respect to the coefficients \(\bm{c}\). With this representation, the
Higham condition number has a closed form since
\begin{equation}
  \left \lVert
  J^{-1} F_{\bm{c}} \bm{\delta} \right \rVert / \eps \leq
  \left \lVert J^{-1} F_{\bm{c}} \right \rVert \left \lVert \bm{\delta} \right
  \rVert / \eps \leq \left \lVert J^{-1} F_{\bm{c}} \right \rVert \left \lVert
  \bm{c} \right \rVert
\end{equation}
for any matrix norm that is compatible with the vector norm used on
\(\bm{\delta}\). The Frobenius matrix norm and vector 2-norm can be
combined to give a condition number that is straightforward to compute:
\begin{equation}\label{eq:higham-cond-num}
  \kappa_H = \left \lVert J^{-1} F_{\bm{c}} \right \rVert_F
  \frac{\left \lVert \bm{c}
  \right \rVert_2}{\left \lVert \bm{\alpha} \right \rVert_2}.
\end{equation}

While this closed form for \(\kappa_H\) is useful, it provides a less sharp
measure than the condition number \(\kappa_{\bm{\alpha}}\) given in
Definition~\ref{defn:abstract-cond-num} since the ball \(\left \lVert
\bm{\delta} \right \rVert_2 \leq \eps \left \lVert \bm{c} \right \rVert_2\)
can allow larger perturbations of a single coefficient than the box determined
by \(\left|\delta_j\right| \leq \eps \left|c_j\right|\) and allows
perturbations in zero coefficients.
When specialized to planar B\'{e}zier curves, we'll show in
Theorem~\ref{thm:kappa-closed-form} that \(\kappa_{\alpha}\) has a closed form
as well\footnote{When the 2-norm is used}. In addition, this closed
form shows that \(\kappa_{\alpha}\) is a natural extension of the
one-dimensional equivalent given in~\eqref{eq:bernstein-cond-num} below.

\subsection{One-dimensional Case}

When \(n = 1\), due to the triangle inequality:
\begin{equation}\label{eq:tri-ineq}
\left|\delta \alpha\right| = \left|J^{-1} \sum_{j = 0}^n
  \delta_j \phi_j(\alpha)\right| \leq \frac{1}{\left|F'(\alpha)\right|}
  \sum_{j = 0}^n \left|\delta_j \phi_j(\alpha)\right|.
\end{equation}
The sign and magnitude of each \(\delta_j\) can be chosen to make
\(\delta_j \phi_j(\alpha) = \left|c_j \phi_j(\alpha)\right| \eps\),
hence for these values equality holds in the triangle inequality:
\begin{equation}
\left|\sum_{j = 0}^n \delta_j \phi_j(\alpha)\right| =
\eps \sum_{j = 0}^n \left|c_j \phi_j(\alpha)\right|.
\end{equation}
Thus we get a root condition number for a polynomial given in Bernstein form
\begin{equation}\label{eq:bernstein-cond-num}
\kappa_{\alpha} =
  \frac{1}{\left|\alpha F'(\alpha)\right|} \sum_{j = 0}^n \left|
  c_j \phi_j(\alpha)\right| =
  \frac{\widetilde{F}(\alpha)}{\left|\alpha F'(\alpha)\right|}
\end{equation}
that agrees with the common definition
(\cite[Equation~12.33]{Farouki2008}) for any polynomial basis
\(\phi_j\).

\section{Conditioning of B\'{e}zier Curve Intersection}

To define a condition number for the intersection of two planar B\'{e}zier
curves, we write the difference as
\begin{equation}
F(s, t) = \left[ \begin{array}{c} x_0(s) \\ 0 \end{array}\right] +
  \left[ \begin{array}{c} 0 \\ y_0(s) \end{array}\right] -
  \left[ \begin{array}{c} x_1(t) \\ 0 \end{array}\right] -
  \left[ \begin{array}{c} 0 \\ y_1(t) \end{array}\right].
\end{equation}
We can show that there is a closed form for the condition number given
by Definition~\ref{defn:abstract-cond-num}, specialized to the
2-norm.

\begin{theorem}\label{thm:kappa-closed-form}
Let \(\left[\begin{array}{c} \alpha \\ \beta \end{array}\right]\) be the
parameter vector \(\left[\begin{array}{c} s \\ t \end{array}\right]\)
at which two planar B\'{e}zier curves \(b_0(s)\) and \(b_1(t)\)
have a transversal intersection.
Then the root condition number of the intersection is
\begin{equation}\label{eq:kappa-max}
\kappa_{\alpha, \beta} = \sqrt{\frac{\mu_1^2
  \left(\bm{v} \cdot \bm{v}\right) +
  2 \mu_1 \mu_2 \left|\bm{v} \cdot \bm{w}\right| +
  \mu_2^2 \left(\bm{w} \cdot \bm{w}\right)}{\alpha^2 + \beta^2}}
\end{equation}
where
\begin{equation}
  J^{-1}\left(\alpha, \beta\right) = \left[ \begin{array}{c c}
  \bm{v} & \bm{w} \end{array}\right] \qquad
  \mu_1 = \widetilde{x}_0(\alpha) + \widetilde{x}_1(\beta) \qquad
  \mu_2 = \widetilde{y}_0(\alpha) + \widetilde{y}_1(\beta).
\end{equation}
Here \(b_0(s) = \left[\begin{array}{c c} x_0(s) & y_0(s)\end{array}\right]^T\),
\(b_1(t) = \left[\begin{array}{c c} x_1(t) & y_1(t)\end{array}\right]^T\)
and the \(\widetilde{x}_i, \widetilde{y}_j\) are as defined in
\eqref{eq:p-tilde}.
\end{theorem}

\begin{proof}
Let the curve \(b_0(s)\) be of degree \(m\) and \(b_1(t)\) be of degree \(n\).
Then \(F(s, t)\) can be written as a sum of \(2(m + 1) + 2(n + 1)\) terms:
\begin{equation}\label{eq:bezier-full-basis}
F(s, t) =
  \sum_{i = 0}^m c_i^{(1)} \left[ \begin{array}{c}
  B_{i, m}(s) \\ 0 \end{array}\right] +
  \sum_{i = 0}^m c_i^{(2)} \left[ \begin{array}{c}
  0 \\ B_{i, m}(s) \end{array}\right] +
  \sum_{j = 0}^n c_j^{(3)} \left[ \begin{array}{c}
  -B_{j, n}(t) \\ 0 \end{array}\right] +
  \sum_{j = 0}^n c_j^{(4)} \left[ \begin{array}{c}
  0 \\ -B_{j, n}(t) \end{array}\right].
\end{equation}
Since \(F(s, t) = b_0(s) - b_1(t)\) we have Jacobian \(J(s, t) =
\left[ \begin{array}{c c} b_0'(s) & -b_1'(t) \end{array}\right]\). Since we
are considering a transversal intersection, we have
\(\det J(\alpha, \beta) \neq 0\). In a perturbed \(F\), we replace each
\(c_j^{(k)}\) with \(c_j^{(k)} + \delta_j^{(k)}\) where
\(\left|\delta_j^{(k)}\right| \leq \eps \left|c_j^{(k)}\right|\).
By writing \(J^{-1} = \left[ \begin{array}{c c}
\bm{v} & \bm{w} \end{array}\right]\), we have
\begin{multline}
J^{-1}\left(\bm{\alpha}\right) \sum_j \delta_j
  \phi_j\left(\bm{\alpha}\right) =
  \left[\sum_{i = 0}^m \delta_i^{(1)} B_{i, m}(\alpha) +
  \sum_{j = 0}^n \delta_j^{(3)} B_{j, n}(\beta)\right] \bm{v} \\
  + \left[\sum_{i = 0}^m \delta_i^{(2)} B_{i, m}(\alpha) +
  \sum_{j = 0}^n \delta_j^{(4)} B_{j, n}(\beta)\right] \bm{w} =
  \nu_1 \bm{v} + \nu_2 \bm{w}
\end{multline}
where
\begin{align}
\left|\nu_1\right| / \eps &\leq \sum_{i = 0}^m
  \left|c_{i}^{(1)}\right| B_{i, m}\left(\alpha\right) + \sum_{j = 0}^n
  \left|c_{j}^{(3)}\right| B_{j, n}\left(\beta\right) =
  \widetilde{x}_0(\alpha) + \widetilde{x}_1(\beta) = \mu_1 \\
\left|\nu_2\right| / \eps &\leq \sum_{i = 0}^m
  \left|c_{i}^{(2)}\right| B_{i, m}\left(\alpha\right) + \sum_{j = 0}^n
  \left|c_{j}^{(4)}\right| B_{j, n}\left(\beta\right) =
  \widetilde{y}_0(\alpha) + \widetilde{y}_1(\beta) = \mu_2.
\end{align}
As in \eqref{eq:tri-ineq}, the bound can be attained by choosing the
sign and magnitude of each perturbation so that
\(\delta_j^{(k)} B_{j, d} = \eps \left|c_j^{(k)}\right| B_{j, d}\).
The factor \(\eps\) can be cancelled to give the relative root
condition number
\begin{align}
\kappa_{\alpha, \beta} &= \frac{1}{\sqrt{\alpha^2 + \beta^2}}
  \sup_{\left|\nu_k\right| \leq \mu_k} \left \lVert \nu_1 \bm{v} +
  \nu_2 \bm{w} \right \rVert_2 \\
  &=
  \sqrt{\frac{\sup_{\left|\nu_k\right| \leq \mu_k}
  \nu_1^2 \left(\bm{v} \cdot \bm{v}\right) +
  2 \nu_1 \nu_2 \left(\bm{v} \cdot \bm{w}\right) +
  \nu_2^2 \left(\bm{w} \cdot \bm{w}\right)}{\alpha^2 + \beta^2}}
  \label{eq:intersect-cond-num}.
\end{align}
Now we seek to maximize the objective function \(\theta(\nu_1, \nu_2) =
\nu_1^2 \left(\bm{v} \cdot \bm{v}\right) +
2 \nu_1 \nu_2 \left(\bm{v} \cdot \bm{w}\right) +
\nu_2^2 \left(\bm{w} \cdot \bm{w}\right)\) in the rectangle
\(\left[-\mu_1, \mu_1\right] \times \left[-\mu_2, \mu_2\right]\).

To find interior critical points, we solve the system \(\theta_{\nu_1} =
\theta_{\nu_2} = 0\):
\begin{equation}
\left[ \begin{array}{c c}
  2 \bm{v} \cdot \bm{v} & 2 \bm{v} \cdot \bm{w} \\
  2 \bm{v} \cdot \bm{w} & 2 \bm{w} \cdot \bm{w} \end{array}\right]
\left[ \begin{array}{c} \nu_1 \\ \nu_2 \end{array}\right] =
\left[ \begin{array}{c} 0 \\ 0 \end{array}\right].
\end{equation}
This system has the unique solution \(\nu_1 = \nu_2 = 0\) unless
\(\|\bm{v}\|_2 \|\bm{w}\|_2 = \left|\bm{v} \cdot \bm{w}\right|\).
By the Cauchy-Schwarz inequality, this can only occur if \(\bm{v}\) and
\(\bm{w}\) are parallel; since \(J^{-1}\) is invertible, we know they
are not. Hence \(\theta(0, 0) = 0\) is the only interior critical point and
it is the global minimum.

Along the boundary of the rectangle,
we fix one of \(\nu_1\) or \(\nu_2\) and the resulting univariate function is
an up-opening parabola. For example, fixing \(\nu_2 = c\) gives
\(\theta(\nu_1, c) =
\nu_1^2 \left(\bm{v} \cdot \bm{v}\right) +
\nu_1\left[2 c \left(\bm{v} \cdot \bm{w}\right)\right] +
c^2 \left(\bm{w} \cdot \bm{w}\right)\) which has positive lead coefficient
\(\|\bm{v}\|_2^2\). The lead coefficient cannot be \(0\) since if \(\bm{v}\)
were the zero vector we would have \(\det J = 0\).
Since \(\theta\) is an up-opening parabola along the boundary, any critical
point must be a local minimum.

Thus we know the maximum occurs at one of the four corners of the
rectangle. Due to sign cancellation, this leads to one of two values
\(\theta =
\mu_1^2 \left(\bm{v} \cdot \bm{v}\right) \pm
2 \mu_1 \mu_2 \left(\bm{v} \cdot \bm{w}\right) +
\mu_2^2 \left(\bm{w} \cdot \bm{w}\right)\), the largest of which is
\(\mu_1^2 \left(\bm{v} \cdot \bm{v}\right) +
2 \mu_1 \mu_2 \left|\bm{v} \cdot \bm{w}\right| +
\mu_2^2 \left(\bm{w} \cdot \bm{w}\right)\). Thus
\begin{equation}\label{eq:intersect-cond-num-too}
\kappa_{\alpha, \beta} = \sqrt{\frac{\mu_1^2
  \left(\bm{v} \cdot \bm{v}\right) +
  2 \mu_1 \mu_2 \left|\bm{v} \cdot \bm{w}\right| +
  \mu_2^2 \left(\bm{w} \cdot \bm{w}\right)}{\alpha^2 + \beta^2}}
\end{equation}
as desired.
\end{proof}

With this closed form \(\kappa_{\alpha, \beta}\) in hand, we can now compare
to the Higham condition number \(\kappa_H\) from~\eqref{eq:higham-cond-num}.
We'll show that \(\kappa_{\alpha, \beta} \leq \kappa_H\) by comparing
\(\kappa_{\alpha, \beta}^2 \left(\alpha^2 + \beta^2\right)\)
to \(\kappa_H^2 \left(\alpha^2 + \beta^2\right) = \left \lVert J^{-1}
F_{\bm{c}} \right \rVert_F^2 \left \lVert \bm{c} \right \rVert_2^2\).
In the case of planar curves with basis as in~\eqref{eq:bezier-full-basis},
\begin{equation}
  \left \lVert J^{-1} F_{\bm{c}} \right \rVert_F^2 =
  \left[\left \lVert \bm{v} \right \rVert_2^2 + \left \lVert \bm{w}
  \right \rVert_2^2 \right] W \quad \text{where} \quad W =
  \sum_{i = 0}^m B_{i, m}^2(\alpha) + \sum_{j = 0}^n B_{j, n}^2(\beta)
\end{equation}
is the sum of squared Bernstein weights.
With two applications of the Cauchy-Schwarz inequality we know that
\begin{equation}
  \kappa_{\alpha, \beta}^2 \left(\alpha^2 + \beta^2\right) \leq
  \left(\mu_1 \left \lVert \bm{v} \right \rVert_2 + \mu_2 \left \lVert \bm{w}
  \right \rVert_2\right)^2 \leq \left(\mu_1^2 + \mu_2^2\right) \left[
  \left \lVert \bm{v} \right \rVert_2^2 + \left \lVert \bm{w}
  \right \rVert_2^2 \right].
\end{equation}
So it remains to show that \(\mu_1^2 + \mu_2^2 \leq W \left \lVert \bm{c}
\right \rVert_2^2\), which can be done with another application of
Cauchy-Schwarz to the terms in \(\mu_1\) and \(\mu_2\):
\begin{equation}
  \mu_1^2 + \mu_2^2 \leq \left(\sum_{i = 0}^m \left|c_i^{(1)}\right|^2 +
  \sum_{j = 0}^n \left|c_j^{(3)}\right|^2 \right) W + \left(\sum_{i = 0}^m
  \left|c_i^{(2)}\right|^2 + \sum_{j = 0}^n \left|c_j^{(4)}\right|^2 \right) W
  = W \left \lVert \bm{c} \right \rVert_2^2.
\end{equation}

\section{Condition Number in Practice}

\subsection{Transversal Intersection}

Consider the line
\(b_0(s) = \left[ \begin{array}{c} 2s \\ 2s \end{array}\right]\)
and quadratically parameterized line
\(b_1(t) = \left[ \begin{array}{c} 4t^2 \\ 2 - 4t^2
\end{array}\right]\) which intersect at \(\alpha = \beta = 1/2\).
At the intersection we have \(J^{-1} = \frac{1}{8}
\left[ \begin{array}{c c} 2 & 2 \\ -1 & 1 \end{array}\right]\),
so that \(\bm{v} \cdot \bm{v} = \bm{w} \cdot \bm{w} =
5/64\) and \(\bm{v} \cdot \bm{w} = 3/64\). Since the
\(x\)-component of \(F(s, t)\) can be written as
\(2s - 4t^2 = 2 B_{1, 1}(s) - 4 B_{2, 2}(t)\) and the
\(y\)-component as \(2s + 4t^2 - 2 = 2 B_{1, 1}(s) - 2 B_{0, 2}(t)
- 2 B_{1, 2}(t) + 2 B_{2, 2}(t)\) we have
\begin{alignat}{2}
\mu_1 &= 2 B_{1, 1}(\alpha) &&+ 4 B_{2, 2}(\beta) = 2 \\
\mu_2 &= 2 B_{1, 1}(\alpha) + 2 B_{0, 2}(\beta) +
  2 B_{1, 2}(\beta) &&+ 2 B_{2, 2}(\beta) = 3.
\end{alignat}
Following~\eqref{eq:intersect-cond-num-too}, this gives
\(\kappa_{\alpha, \beta} = \sqrt{202}/8 \approx 1.78\).
This low condition number is expected from a geometric point of view; i.e.
the intersection is a transversal intersection of two lines. However,
when using the resultant to eliminate each parameter, one of the two roots is
a double root:
\begin{align}
\operatorname{Res}_t\left(x_0(s) - x_1(t), y_0(s) - y_1(t)\right) &=
  64(2s - 1)^2 \\
\operatorname{Res}_s\left(x_0(s) - x_1(t), y_0(s) - y_1(t)\right) &=
  4(2t - 1)(2t + 1).
\end{align}
so an algebraic approach may lead to an incorrect conclusion that the
intersection is ill-conditioned.

\subsection{Collapsing to One-dimensional Case}

One key argument for choosing Definition~\ref{defn:abstract-cond-num} over
the Highham condition number \(\kappa_H\) from~\eqref{eq:higham-cond-num} is
that \(\kappa_{\alpha, \beta}\) is a natural extension of the condition number
for the equivalent one-dimensional problem. To see that this is so, we'll
define a ``trivial" example by starting with a polynomial \(p(s)\) in Bernstein
form and a simple root \(\alpha\).

We define the B\'{e}zier curves \(b_0(s) = \left[ \begin{array}{c} p(s) \\ 0
\end{array}\right]\) and \(b_1(t) = \left[ \begin{array}{c} 0 \\ t
\end{array}\right]\). These curves intersect when \(\beta = 0\) and
\(\alpha\) is a root of \(p(s)\).
At such an intersection \(\mu_1 = \widetilde{p}(\alpha) + 0\),
\(\mu_2 = 0 + \beta = 0\) and
\begin{equation}
J^{-1} = \left[ \begin{array}{c c} 1 / p'(\alpha) & 0 \\ 0 & -1 \end{array}\right]
\end{equation}
so that \(\bm{v} \cdot \bm{v} = 1 / \left[p'(\alpha)\right]^2\), \(\bm{v} \cdot
\bm{w} = 0\) and \(\bm{w} \cdot \bm{w} = 1\). This produces
\begin{equation}
  \kappa_{\alpha, 0} = \sqrt{\frac{\mu_1^2 \left(\bm{v} \cdot \bm{v}\right) +
  0 + 0}{\alpha^2 + 0}} = \frac{\widetilde{p}(\alpha)}{\left|\alpha
  p'(\alpha)\right|},
\end{equation}
the commonly used condition number presented
in~\eqref{eq:bernstein-cond-num}.

\subsection{Line-line Intersection with Poorly Behaved Coefficients}

Consider the intersection of the lines \(y = x\) and \(y = 1 - x\) when
\(x \in \left[0, 1\right]\). These correspond to the B\'{e}zier curves
\begin{equation}
b_0(s) = \left[ \begin{array}{c} 1 \\ 1 \end{array}\right] s, \quad
b_1(t) = \left[ \begin{array}{c} 0 \\ 1 \end{array}\right] (1 - t) +
\left[ \begin{array}{c} 1 \\ 0 \end{array}\right] t.
\end{equation}
By adding a scalar \(D > 0\) to each component, we leave \(F(s, t)\) and hence
the solution unchanged. However, the coefficients of the curves change:
\begin{equation}
b_0(s) = \left[ \begin{array}{c} D(1 - s) + (1 + D)s \\ D(1 - s) + (1 + D)s
  \end{array}\right], \quad b_1(t) = \left[ \begin{array}{c}
  D(1 - t) + (1 + D)t \\ (1 + D)(1 - t) + Dt \end{array}\right].
\end{equation}
At the solution \(\alpha = \beta = 1/2\), we have \(\mu_1 = \mu_2 = 2D + 1\)
and
\begin{equation}
J^{-1} = \frac{1}{2}
\left[ \begin{array}{c c} 1 & 1 \\ -1 & 1 \end{array}\right]
\end{equation}
so that \(\bm{v} \cdot \bm{v} = \bm{w} \cdot \bm{w} =
1/2\) and \(\bm{v} \cdot \bm{w} = 0\).
So, we see the condition number \(\kappa_{\alpha, \beta} = \sqrt{2}(2D + 1)\)
increases towards infinity as \(D\) does. This is what we expect as the
coefficients grow so large that their ratio \((1 + D) / D\) approaches \(1\).

\subsection{Family of Lines Approaching Coincidence}

Consider a family of intersections in which one of the lines approaches the
other:
\begin{equation}
b_0(s) = \left[ \begin{array}{c} s \\ 1 \end{array}\right], \quad
b_1(t) = \left[ \begin{array}{c} t \\ (1 + r)(1 - t) + t \end{array}\right].
\end{equation}
These lines \(y = 1\) and \(rx + y = 1 + r\) intersect when
\(\alpha = \beta = 1\). However as \(r \longrightarrow 0^+\), the lines
become coincident: if \(r = 0\) the single intersection becomes infinitely
many.

At the solution, we have \(\mu_1 = \mu_2 = 2\) and
\begin{equation}
J^{-1} = \frac{1}{r}
\left[ \begin{array}{c c} r & 1 \\ 0 & 1 \end{array}\right]
\end{equation}
so that \(\bm{v} \cdot \bm{v} = 1\), \(\bm{v} \cdot \bm{w} =
1/r\) and \(\bm{w} \cdot \bm{w} = 2/r^2\). Again we have a condition number
\begin{equation}
\kappa_{\alpha, \beta} = \sqrt{\frac{4}{r^2} + \frac{4}{r} + 2} =
  \frac{2}{r} + 1 + \frac{r}{4} + \bigO{r^2}.
\end{equation}
that increases towards infinity as the parameter \(r \longrightarrow 0^+\).

\section{Conclusion and Future Work}

The author hopes that this can be useful for evaluating and comparing the
performance of curve intersection implementations. By establishing a
straightforward and easy to compute closed form, the condition number can be
used more often to differentiate between cases where the algorithm or computer
code is at fault for loss in accuracy and cases where the conditioning of the
intersection itself is the cause. The framework set forth in
Section~\ref{sec:generic} can be applied in future work to compute the
condition number of higher order intersections such as surface-surface
intersections in \(\reals^3\).

\section{Acknowledgements}

The author would like to thank Professor Tom Sederberg for his useful
course notes (\cite{SederbergNotes}) that provide a useful introduction to
intersection methods as well as the survey \cite{Sederberg1986}. Additionally
\cite{Farouki1987} served as a great baseline when considering conditioning
of operations on B\'{e}zier curves.

%% H/T: https://tex.stackexchange.com/a/137379/32270
\section*{\refname}
\bibliography{paper}

\end{document}